\documentclass[a4paper]{amsart}
\newtheorem{thm}{Theorem}[section]
\newtheorem{proposition}[thm]{Proposition}
\theoremstyle{definition}

\theoremstyle{remark}

\begin{document}

%\catchline{}{}{}{}{}
%
%%%%%%%%%%%%%%%%%%%%%%%%%%%%%%%%%%%%%%%%%%%%%%%%%%%%%%%%%%%%%%%%%%%%

\title{INVERSION OF A MAPPING ASSOCIATED\\
 WITH THE AOMOTO-FORRESTER SYSTEM}

\author{Raquel Caseiro}
\address{CMUC, Department of Mathematics, University of Coimbra
 \\ 3001-454 Coimbra, Portugal}
\email{raquel@mat.uc.pt}
\thanks{Raquel Caseiro was supported by PTDC/MAT/099880/2008
and  Centro de Matem\'{a}tica da
Universidade de Coimbra (CMUC), funded by the European Regional
Development Fund through the program COMPETE and by the Portuguese
Government through the FCT - Funda\c{c}\~{a}o para a Ci\^{e}ncia e a Tecnologia
under the project PEst-C/MAT/UI0324/2011. }
\author{Jean-Pierre Fran\c{c}oise}
\address{Universit\'e P.-M. Curie, Laboratoire J.-L. Lions, 4 Pl. Jussieu\\
Paris, 75005, France}
\email{Jean-Pierre.Francoise@upmc.fr }
\author{Ryu Sasaki}
\address{Yukawa Institute for Theoretical Physics, Kyoto University\\
Kyoto, 606-8502, Japan}
\email{ryu@yukawa.kyoto-u.ac.jp }

%\begin{history}
%\received{(Day Month Year)}
%\revised{(Day Month Year)}
%\end{history}

\begin{abstract}
This article is devoted to the study of a general class of Hamiltonian systems which extends the Calogero systems with external quadratic potential associated to any root system. The interest for such a class comes from a previous article of Aomoto and Forrester. We consider first the one-degree of freedom case and compute the Birkhoff series defined near each of its stationary points. In general, the analysis of the system motivates finding some expression for the inverses of a rational map introduced by Aomoto and Forrester. We derive here some diagrammatic expansion series for these inverses.
\end{abstract}

%%% ----------------------------------------------------------------------
\maketitle
%%% --------------------------------------

\keywords{Birkhoff normal form; Residue; Arrangement of hyperplanes; Hamilton-Jacobi equation; Diagrammatic Expansions.}

%\subjclass[2000]{Mathematics Subject Classification 2010: 37J05, 37J30}

\section{Introduction}
In \cite{ga-ma73}, Gallavotti and Marchioro derived a remarkable integral formula as a consequence of Feynmann-Kac path integral method applied to the quantum Calogero system. They posed the question of finding a purely classical proof. In \cite{fr88} the question was answered and the formula was derived by showing the existence of an associated symplectic action of the torus. In \cite{ao-fo00} an integral was computed for a much general setting of a Hamiltonian defined on the complement of an arrangement of hyperplanes. In this article, we call the Aomoto-Forrester Hamiltonian system, the system defined by:

\begin{equation}
H=\frac{1}{2}\sum_{j=1}^n\left\{y_j^2+\left[x_j-\sum_{h\in A}\frac{\lambda_h u_{h,j}}{<u_h,x>+u_{h,0}}\right]^2\right\},\quad
\omega=\sum_{j=1}^n dx_j{\wedge}dy_j, \lambda_h>0.
\end{equation}

The index $h$ enumerates an arrangement of {\it real} hyperplanes defined by the equations:
\begin{equation}
f_h(x):=<u_h,x>+u_{h,0}=u_{h,0}+\sum_{j=1}^n u_{h,j}x_j=0.
\end{equation}
In the particular case where the arrangement of hyperplanes defines the walls of Weyl chambers,
and hence where the complement of these hyperplanes is invariant under the action of a Weyl group, we recover the class of Hamiltonians considered by (\cite{bo-co-sa98}, \cite{bo-co-sa99}, \cite{bo-ma-sa00}, \cite{kh-po-sa00}, \cite{kh-sa01}) (both in quantum and classical cases).
We call this situation, the root system cases. For these root systems, the integral formula was also proved via the existence of an associated symplectic action of the torus in (\cite{ca-fr-sa00}, \cite{ca-fr-sa01}).
In contrast, the Aomoto-Forrester integral formula was proved by a complex analytic residue techniques (Griffith's residue).
It is remarkable that it covers a much more general situation. For instance, this one-dimensional system:
\begin{equation}
\begin{array}{l}
\displaystyle H=\frac{1}{2}y^2+\frac{1}{2}\left(x-\sum_{h=1}^N \frac{\lambda_h}{x-a_h}\right)^2,\\
\displaystyle \omega=dx{\wedge}dy,\, \lambda_h>0,
\end{array}
\end{equation}
provides already a quite meaningful example. It is known that if a ($n$-degrees of freedom) Hamiltonian System is associated to the symplectic action of an $n$-torus, it displays critical points of Birkhoff type and the Birkhoff normal form is reduced to linear terms.
\vskip 1pt
After the Aomoto-Forrester article, it was natural to ask the question of the existence of an associated symplectic action of the torus for the Aomoto-Forrester system.
This is rather easily disproved here with the one-dimensional system (1.3) defined above.
The system (1.3) displays $N+1$ critical points of Morse type. We are indeed able to derive a formula for the inverse of the Birkhoff series of each of these points using a Cauchy residue (Lagrange inversion formula).
We check on specific examples that in general these series do not reduce to linear terms.
\vskip 1pt
In a second part, we obtain partial results in the general case of the Aomoto-Forrester system. This system is likely to be non-integrable in general. Nevertheless, using some techniques (\cite{ab74}, \cite{ba-co-wr82},
\cite{zh05}) introduced around the Jacobian conjecture, we can derive some expansion for the inverses of the Aomoto-Forrester map.

\section{The inverse of the Birkhoff series in the one-dimensional case and the Lagrange inversion formula}
In this paragraph, we focus on the one-dimensional Hamiltonian:

\begin{equation}
H=\frac{1}{2}y^2+\frac{1}{2}\left(x-\sum_{j=1}^N \frac{\lambda_j}{x-a_j}\right)^2.
\end{equation}
This Hamiltonian is defined on bands $\left]a_i,a_{i+1}\right[$.
Under the condition $\lambda_j>0$, this HS displays a unique stationary point $(b_i,0), i=0,...,N$ on each interval  $\left]a_i,a_{i+1}\right[$ and these stationary points are minima and of Birkhoff type.
To each of these points one can associate a Birkhoff normal form. Birkhoff series are convergent for the systems with one degree of freedom. The convergence of the Birkhoff normal form for the systems with one degree of freedom was essentially shown by Siegel \cite{sm}.

In the coordinate system $(y,w)$, the system displays
\begin{equation}
\begin{array}{l}
\displaystyle H= \frac{1}{2}y^2+\frac{1}{2}w^2,\quad w=x-\sum_{j=1}^N\frac{\lambda_j}{x-a_j},\\
 \displaystyle \omega=dx{\wedge}dy.
 \end{array}
 \end{equation}
Write
\begin{equation}
w=f(x)=\frac{q(x)}{p(x)},\quad
q(x)=\prod_{k=0}^{N}(x-b_k), \quad p(x)=\prod_{j=1}^N(x-a_j).
\end{equation}
Translation of coordinates $x=b_k+\xi$ yields:
\begin{equation}
\begin{array}{l}
\displaystyle H=\frac{1}{2}y^2+\frac{1}{2}f(b_k+\xi)^2=
\frac{1}{2}y^2+\frac{1}{2}f'(b_k)^2\xi^2+O(\xi^3).
\end{array}
\end{equation}
Changing coordinates $X=f'(b_k)\xi$ rescale, locally near the critical point $(b_k,0)$ the Hamiltonian $H$ to the Morse function
\begin{equation}
H=\frac{1}{2}y^2+\frac{1}{2}X^2+...
\end{equation}
The $1$-form $\eta=xdy$ is cohomologous to $\frac{1}{f'(b_k)}Xdy$.
Let $\gamma=H^{-1}(c)$ be a circle of radius $c$ centered at $0$, positively oriented and parameterized by:
\begin{equation}
\gamma=\{y=\sqrt{2c}\,\sin\theta, w=\sqrt{2c}\,\cos\theta, \theta\in \left[0,2\pi\right[\}.
\end{equation}
As the variable $(y,w)$ circles around $\gamma$, the points $(x_k(w),y)$ where $x_k(w)$ is a root of the polynomial equation $w=f(x)$ circles around a vanishing cycle of the Morse critical point $(b_k,0)$.

Relative cohomology techniques  developed around the isochore Morse lemma (\cite{fr78}, \cite{fr98}, \cite{fr-ga-ga10}, \cite{fr-ga-ga13}) show that the derivative of the (inverse of) Birkhoff series is proportional to the period integral:
\begin{equation}
\int_{\gamma} x_k(w)dy,
\end{equation}
where $x_k(w)$  denotes the unique series obtained by inverting $w=f(x)$ near the point $(b_k,0)$. This shows that the computation of the (inverse of) Birkhoff series is essentially given by the inversion of the mapping $x\mapsto w=f(x)$.
This inverse series can be obtained in many different ways. Here we focus on the Lagrange inversion formula:

\begin{proposition}
Let
\begin{equation}
g_i(z)= \frac{\partial^{i-1}}{\partial z^{i-1}}\left\{\left[\frac{z-b_k}{f(z)}\right]^i\right\},
\end{equation}
then
\begin{equation}\label{eq:Lagrange:inversion}
x_k(w)=b_k+\sum_{i=1}^{+\infty} g_i(b_k)\frac{w^i}{i!}.
\end{equation}
\end{proposition}

\begin{proof}
Let $\displaystyle \gamma_0$ be a closed curve around $0$ and $\displaystyle\gamma_k=x_k(\gamma_0)$ such that $f$ has no more singularities inside $\gamma_k$ than $b_k$.
 Cauchy integral formula reads,
\begin{eqnarray*}
x_k(w)&=&\displaystyle\frac{1}{2\sqrt{-1}\pi}\int_{\gamma_k}\frac{x_k(z)}{z-w}dz.
\end{eqnarray*}
Making $z=f(x)$ we have
\begin{eqnarray*}
x_k(w)&=&\displaystyle \frac{1}{2\sqrt{-1}\pi}\int_{\gamma_k}\frac{x f'(x)}{f(x)-w}dx\\
&=&\displaystyle \frac{1}{2\sqrt{-1}\pi}\int_{\gamma_k}x\frac{f'(x)}{f(x)}\frac{1}{1-\frac{w}{f(x)}}dx\\
&=&\displaystyle \frac{1}{2\sqrt{-1}\pi}\sum_{i\geq 0}\int_{\gamma_k}x\frac{f'(x)}{f(x)^{i+1}}w^idx\\
&=&\displaystyle\frac{1}{2\sqrt{-1}\pi}\int_{\gamma_k}x\frac{f'(x)}{f(x)}dx + \frac{1}{2\sqrt{-1}\pi}\sum_{i\geq 1}w^i\int_{\gamma_k}x\frac{f'(x)}{f(x)^{i+1}}dx.
\end{eqnarray*}
By residue formula,
\begin{eqnarray*}
x_k(w)&=& \displaystyle b_k + \sum_{i\geq 1}w^i{\rm Res}_{x=b_k}(x-b_k)\frac{f'(x)}{f(x)^{i+1}}\\
&=& \displaystyle b_k + \sum_{i\geq 1}\frac{w^i}{i}{\rm Res}_{x=b_k}\frac{1}{f(x)^i}\\
&=& \displaystyle b_k + \sum_{i\geq 1} \frac{w^i}{i} \, \frac{1}{(i-1)!}\frac{\partial^{i-1}}{\partial x^{i-1}}\left.\left[\frac{(x-b_k)^i}{f(x)^i}\right]\right|_{x=b_k},
\end{eqnarray*}
and Lagrange integral formula follows.
\end{proof}

%It reads as follows:
%
%Denote
%\begin{equation}
%g_i(z)= \frac{\partial^{i-1}}{\partial z^{i-1}}\left\{\left[\frac{z-b_k}{f(z)}\right]^i\right\},
%\end{equation}
%then
%\begin{equation}
%x_k(w)=b_k+\sum_{i=1}^{+\infty} g_i(b_k)\frac{w^i}{i!}.
%\end{equation}

Note that the closed formula derived above for the coefficients of this series can be easily implemented. We did it with Mathematica.
 First two terms $g_1(b_k)$ and $g_2(b_k)$ are computed below:
\begin{align}
 g_1(b_k)&=\frac{\prod_{j=1}^M (b_k-a_j)}{\prod_{j=0,j\neq k}^M (b_k-b_j)},\\
 g_2(b_k)&=2\sum_{l=1,l\neq k}^M\left\{-\frac{(b_k-a_l)^2\cdot \prod_{j=1,j\neq l}^M(b_k-a_j)^2}{(b_k-b_l)^3\cdot\prod_{j=0,j\neq k,l}^M (b_k-b_j)^2}\right.\nonumber\\
&\left. \qquad \qquad \quad \ +\frac{(b_k-a_l)\cdot \prod_{j=1,j\neq l}^M(b_k-a_j)^2}{(b_k-b_l)^2\cdot\prod_{j=0,j\neq k,l}^M (b_k-b_j)^2}\right\}.
\end{align}

$g_3(b_k)$ is already quite complicated and the general form might not be suitable for printing.

The final step is to include this formula in the integral:

\begin{equation}
\int_{\gamma=H^{-1}(c)} x_k(w)dy.
\end{equation}
This yields ultimately:
\begin{equation}
\int_{\gamma=H^{-1}(c)} x_k(w)dy=\sum_i \int_{0}^{2\pi} \frac{1}{i!}g_i(b_k)c^{i+1}(\cos\theta)^{i+1}d\theta
\end{equation}
\begin{equation}
=\sum_p \frac{1}{(2p-1)!}g_{2p-1}(b_k)c^{2p}\int_0^{2\pi} \cos^{2p}\theta d\theta=
\sum_p 2\pi\frac{2p}{(2^pp!)^2}g_{2p-1}(b_k)c^{2p}.
\end{equation}

%%%%%%%%%%%%%%%%%%%%%%%%%%%%% The general case %%%%%%%%%%%%%%%%%%%%%%%%

\section{The general case and the inverses of the Aomoto-Forrester map}

Let $A$ be a finite arrangement of hyperplanes in the $n$-dimensional complex affine space $\mathbb{C}^n$. Let $N(A)$ be the union of hyperplanes of $A$ in $\mathbb{C}^n$ and $M(A)$ be its complement. It is further assumed that $A$ is real, meaning that the defining function of every hyperplane $f_h(z)=u_{h,0}+\sum_{i=1}^n u_{h,i}z_i$ has real coefficients. We define the Aomoto-Forrester system on $T^*(M(A)\cap\mathbb{R}^n)$ by the Hamiltonian function:
\begin{equation}
H=\frac{1}{2}\sum_{i=1}^n y_i^2+\frac{1}{2}\sum_{i=1}^n\left(x_i-\sum_{h\in A}\frac{\lambda_h u_{h,i}}{f_h(x)}\right)^2.
\end{equation}
The isochore geometry of this Hamiltonian system is, in particular, devoted to the study of the canonical partition function
\begin{equation}
\int {\rm exp}({-\beta H})\Omega,\quad \Omega=\left(dx{\wedge}dy\right)^n.
\end{equation}
In the purpose of computing this function, Aomoto and Forrester introduced the mapping:
\begin{equation}
F:z\mapsto w, \quad w_i=z_i-\sum_{h\in A} \frac{\lambda_h u_{h, i}}{f_h(z)}.
\end{equation}
A connected component of $M(A)\cap\mathbb{R}^n$ is called a chamber. Aomoto and Forrester proved that the mapping $F$ displays in restriction to any chamber $\Delta$ a unique analytic inverse. Note that this map displays a special form which seems to have not been exploited before. It writes:
\begin{equation}
F: z\mapsto w,\quad  w_i=z_i-\frac{\partial P}{\partial z_i},
\end{equation}
with:
\begin{equation}
P={\rm log}\left(\prod_{h\in A} f_h(z)^{\lambda_h}\right).
\end{equation}
It is possible to show, for instance by using the Abyankhar formula for the inverse of a map (\cite{ab74}, \cite{ba-co-wr82}) that the inverse of a map of this type is also a map of this type:
\begin{equation}
G: w\mapsto z, z_i= w_i+\frac{\partial Q}{\partial w_i}.
\end{equation}
Inspired on the recent contributions to the Jacobian conjecture (\cite{zh05}), we propose to characterize these inverses as follows. We consider for that purpose the mapping
\begin{equation}
F_t: z\mapsto F_t(z)=z-t{\nabla}P.
\end{equation}
For $t$ small, this mapping is invertible. We give an expansion for this inverse which exists for $t$ small.
Aomoto-Forrester showed that the map $F_t$ is invertible till $t=1$ when it is restricted to a chamber $\Delta$.
These local inverses are local restrictions of extensions of the inverse we described below till $t=1$.
Note that there is no apparent mechanical meaning of the multiple $t$. It can be seen as a simultaneous scaling of the parameters $\lambda_h$.

\begin{thm}
Let $t$ be a small parameter, consider the deformation $F_t(z)=z-t{\nabla}P$. The inverse map of $z\mapsto F_t(z)$ can be written:
\begin{equation}
G_t(z)=z+t{\nabla}Q_t(z),
\end{equation}
where $Q_t(z)$ is the unique solution of the Cauchy problem for the Hamilton-Jacobi equation:
\begin{equation}
\begin{array}{l}
\displaystyle \frac{\partial Q_t(z)}{\partial t}= \frac{1}{2}\langle{\nabla}Q_t, {\nabla}Q_t\rangle,\\
\displaystyle Q_{t=0}(z)=P(z).
\end{array}
\end{equation}
\end{thm}

We include a proof to be self-contained although it follows the lines of (\cite{zh05}). Note that with $U_t={\nabla}Q_t$, the equation can be alternatively written:
\begin{equation}
\frac{\partial U_t(z)}{\partial t}=J(U_t(z))U_t(z),
\end{equation}
where $J$ is the Jacobian matrix, which is the inviscid $n$-dimensional Burgers equation (cf. \cite{zh05}).

\begin{proof}
More generally, the formal inverse of $F_t(z)=z-tM(z)$ is the formal series $G_t(z)=z+tN_t(z)$ if and only if:
\begin{equation}
\begin{array}{l}
\displaystyle N_t(F_t(z))=M(z),\\
\displaystyle M(G_t(z))=N_t(z).
\end{array}
\end{equation}
This yields the equations:

\begin{equation}
0=\frac{\partial}{\partial t}[N_t(F_t(z))],
\end{equation}

\begin{equation}
0=\frac{\partial N_t}{\partial t}(F_t(z))+J(N_t)(F_t(z)))\frac{\partial F_t}{\partial t},
\end{equation}

\begin{equation}
0=\frac{\partial N_t}{\partial t}(F_t(z))-J(N_t)(F_t(z)))M.
\end{equation}
After composing with $G_t(z)$ from the right, this displays:
\begin{equation}
\frac {\partial N_t}{\partial t}=J(N_t)M(G_t)=J(N_t)N_t.
\end{equation}
The theorem is proved in particular in the gradient case where, $M={\nabla}P$, $N_t={\nabla}Q_t$,  ${\nabla}Q_t(F_t)={\nabla}P$ and ${\nabla}P(G_t)={\nabla}Q_t$ as a consequence of:

\begin{equation}
\frac{\partial}{\partial z_i}\left(\frac{\partial Q_t}{\partial t}\right)=\sum_j \frac{\partial^2 Q_t}{\partial z_i \partial z_j}\frac{\partial Q_t}{\partial z_j}
=\frac{\partial}{\partial z_i}\frac{1}{2} \langle{\nabla} Q_t, {\nabla} Q_t\rangle.
\end{equation}

\end{proof}
 This yields an expansion series for the inverse map obtained by solving the Hamilton-Jacobi equation (3.8) as follows. Write:

 \begin{equation}
 \begin{array}{l}
 \displaystyle Q_t(z)=\sum_{m=0}^{+\infty} t^{m}Q^{[m]}(z),\\
 \displaystyle Q^{[0]}(z)=P(z),
 \end{array}
 \end{equation}
 and identify terms of equal powers in $t$ displays the recurrency relation:

 \begin{equation}
 mQ^{[m]}(z)=\frac{1}{2}\sum_{i+j=m-1} \langle{\nabla} Q^{[i]}, {\nabla} Q^{[j]}\rangle{(z)}.
 \end{equation}

 Introduce at this point the notation $u_h$ for the $n$-dimensional vector $u_{h,i}, i=1,...,n$ where the index $h$ enumerates the hyperplanes of the arrangement $A$. We find successively:

 \begin{equation}
 Q^{[1]}(z)=\frac{1}{2}\sum_{(h,k)\in A{\times}A} \frac{\lambda_h\lambda_k \langle u_h, u_k \rangle}{f_h(z)f_k(z)},
 \end{equation}

\begin{equation}
 Q^{[2]}(z)=-\frac{1}{2}\sum_{(h,k,l)\in A{\times}A{\times}A} \frac{\lambda_h\lambda_k\lambda_l \langle u_h, u_k \rangle\langle u_h, u_l \rangle}{f_k(z)f_h^2(z)f_k(z)}.
\end{equation}

Each function $Q^{[m]}$, $m\geq 1$, may be represented by a graph. Each connected component of the graph is a non-labeled tree and
represents a different summand of $Q^{[m]}$. All the trees have $m+1$ vertices connected by $m$ edges.
At a vertex $h_i$ a factor $\lambda_{h_i}$ is attached and  an edge connecting  vertices $h_i$ and $h_j$
gives a  factor $\displaystyle \frac{\langle u_{h_i},u_{h_j}\rangle}{f_{h_i}(z)f_{h_j}(z)}$
to the corresponding summand. The whole expression is summed over $\sum_{(h_1,\ldots,h_{m+1})\in A^{m+1}}$. %Since  on $Q_m$  the summation is over all hyperplanes of  arrangement $A$,  trees are non-labeled.
With this representation we see that:

%%%%%% Diagrams %%%%%%%
\begin{eqnarray*}
%%%%%% Q_1 %%%%%
Q^{[1]}&:&
 \frac12\begin{picture}(15,12)
\put(2,3){\circle{1}} \put(12,3){\circle{1}}  \put(2,3){\line(1,0){10}}
\end{picture}\\
%%%%% Q^{[2]} %%%%%%%
Q^{[2]}&:&%\frac12\eval{\nabla Q^{[0]},\nabla Q^{[1]}}=
-\frac12\begin{picture}(25,13)
\put(2,3){\circle{1}} \put(2,3){\line(1,0){10}} \put(12,3){\circle{1}} \put(12,3){\line(1,0){10}}  \put(22,3){\circle{1}}
\end{picture}\\
%%%%% Q_3 %%%%%%%
Q^{[3]}&:&\frac13
\begin{picture}(25,13)
\put(2,3){\circle{1}} \put(2,3){\line(1,0){10}} \put(12,3){\circle{1}} \put(12,3){\line(1,0){10}}  \put(22,3){\circle{1}}
\put(12,3){\line(0,1){10}}  \put(12,13){\circle{1}}
\end{picture}
+ \frac12 \begin{picture}(45,13)
\put(2,3){\circle{1}} \put(2,3){\line(1,0){10}} \put(12,3){\circle{1}} \put(12,3){\line(1,0){10}}  \put(22,3){\circle{1}}
\put(22,3){\line(1,0){10}}  \put(32,3){\circle{1}}
\end{picture}\\
%%%%%% Q_4 %%%%%%
Q^{[4]} &:& -\frac12\begin{picture}(45,13)
\put(2,3){\circle{1}} \put(12,3){\circle{1}}  \put(22,3){\circle{1}} \put(32,3){\circle{1}} \put(42,3){\circle{1}}
\put(2,3){\line(1,0){10}} \put(12,3){\line(1,0){10}} \put(22,3){\line(1,0){10}} \put(32,3){\line(1,0){10}}
\end{picture}
- \begin{picture}(35,14)
\put(2,3){\circle{1}} \put(12,3){\circle{1}}  \put(22,3){\circle{1}} \put(12,13){\circle{1}}
\put(32,3){\circle{1}} \put(2,3){\line(1,0){10}} \put(12,3){\line(1,0){10}} \put(22,3){\line(1,0){10}} \put(12,3){\line(0,1){10}}
\end{picture}
-\frac14 \begin{picture}(25,13)
\put(2,3){\circle{1}} \put(12,3){\circle{1}}  \put(22,3){\circle{1}} \put(12,13){\circle{1}}
\put(12,-7){\circle{1}} \put(2,3){\line(1,0){10}} \put(12,3){\line(1,0){10}} \put(12,3){\line(0,1){10}}
\put(12,2){\line(0,-1){10}}
\end{picture}\\
%%%%%%% Q^{[5]} %%%%%%
Q^{[5]} &:&
\frac{1}{2}\begin{picture}(55,13)
\put(2,3){\circle{1}} \put(12,3){\circle{1}}  \put(22,3){\circle{1}} \put(32,3){\circle{1}} \put(42,3){\circle{1}}
\put(52,3){\circle{1}} \put(2,3){\line(1,0){10}} \put(12,3){\line(1,0){10}} \put(22,3){\line(1,0){10}} \put(32,3){\line(1,0){10}}
\put(42,3){\line(1,0){10}}
\end{picture}
+\begin{picture}(45,14)
\put(2,3){\circle{1}} \put(12,3){\circle{1}}  \put(22,3){\circle{1}} \put(12,13){\circle{1}} \put(32,3){\circle{1}}
\put(42,3){\circle{1}}  \put(2,3){\line(1,0){10}} \put(12,3){\line(1,0){10}} \put(22,3){\line(1,0){10}} \put(32,3){\line(1,0){10}}
\put(12,3){\line(0,1){10}}
\end{picture}
+\begin{picture}(45,14)
\put(2,3){\circle{1}} \put(12,3){\circle{1}}  \put(22,3){\circle{1}}   \put(22,13){\circle{1}} \put(32,3){\circle{1}}
\put(42,3){\circle{1}}  \put(2,3){\line(1,0){10}} \put(12,3){\line(1,0){10}} \put(22,3){\line(1,0){10}} \put(32,3){\line(1,0){10}}
\put(22,3){\line(0,1){10}}
\end{picture}
+ \frac12\begin{picture}(35,14)
\put(2,3){\circle{1}} \put(12,3){\circle{1}}  \put(22,3){\circle{1}} \put(12,13){\circle{1}} \put(22,13){\circle{1}}
\put(32,3){\circle{1}}
 \put(2,3){\line(1,0){10}} \put(12,3){\line(1,0){10}} \put(22,3){\line(1,0){10}}  \put(12,3){\line(0,1){10}}
 \put(22,3){\line(0,1){10}}
\end{picture}
+\begin{picture}(35,28)
\put(2,3){\circle{1}} \put(12,3){\circle{1}}  \put(22,3){\circle{1}} \put(12,13){\circle{1}} \put(32,3){\circle{1}}
 \put(12,-7){\circle{1}} \put(2,3){\line(1,0){10}} \put(12,3){\line(1,0){10}} \put(22,3){\line(1,0){10}}
 \put(12,3){\line(0,1){10}}\put(12,2){\line(0,-1){10}}
\end{picture}
+\frac15\begin{picture}(25,26)
\put(2,3){\circle{1}} \put(12,3){\circle{1}}  \put(22,3){\circle{1}} \put(2,13){\circle{1}}  \put(22,13){\circle{1}}
\put(12,-7){\circle{1}} \put(2,3){\line(1,0){10}} \put(12,3){\line(1,0){10}} \put(12,3){\line(-1,1){10}} \put(12,3){\line(1,1){10}}
\put(12,2){\line(0,-1){10}}
\end{picture}
\end{eqnarray*}

Further analysis of this diagrammatic expansion will be developed in relation with finding integrable cases in another publication.
\section{Conclusion}
To conclude, we have seen in this article that the (inverse of) Birkhoff series of the $1$-dimensional Aomoto-Forrester system is essentially given by the inverses of the Aomoto-Forrester mapping $x\mapsto w=f(x)$.
In the $1$-dimensional case, it is given by Lagrange's inversion formula. In the general case, we cannot consider anymore the Birkhoff series (which may be divergent) but we can still obtain some results on the inverses of the Aomoto-Forrester mapping. We have focussed in this article on the use of Zhao's formula based on solving a Hamilton-Jacobi equation. In \cite{zh11}, the links with other formulas for the inverse (like the Abhyankar-Gurjar formula) are explained. In the $1$-dimensional case, the Abhyankar-Gurjar formula and the Lagrange formula coincides. A recent discussion of the links between perturbation theory of classical Hamiltonian Systems, QFT and the Lagrange inversion formula has been made by G. Gallavotti \cite{G13}.

\vfill\eject

\end{document}